\documentclass[11.50pt,reqno]{amsart} 

\usepackage{amsmath, amssymb} 
\usepackage{amscd}            
\usepackage{amsxtra}          
\usepackage{upref}            

\usepackage[mathscr]{eucal}

\theoremstyle{plain}

\newtheorem{thm}{Theorem}[section]
\newtheorem{lem}[thm]{Lemma}

\newtheorem{prop}[thm]{Proposition}

\numberwithin{equation}{section}
\newcommand{\be}%
  {\protect\setcounter{equation}{\value{subsubsection}}}  
\newcommand{\ee}%
  {\protect\setcounter{subsubsection}{\value{equation}}}

\begin{document}

\title{Generalizations of Jacobsthal Sums and hypergeometric series over finite fields}

\author{Pramod Kumar Kewat and Ram Kumar}
\address{Department of Applied Mathematics\\ Indian Institute of Technology (ISM), Dhanbad-826004\\ Jharkhand, India}
\email{pramodk@iitism.ac.in, ramkumarbhu1991@gmail.com}

\begin{abstract}
For non-negative integers $l_{1}, l_{2},\ldots, l_{n}$, we define character sums $\varphi_{(l_{1}, l_{2},\ldots, l_{n})}$
and $\psi_{(l_{1}, l_{2},\ldots, l_{n})}$ over a finite field which are 
generalizations of Jacobsthal and modified Jacobsthal sums, respectively. We express these character sums in terms of Greene's finite 
field hypergeometric series. We then express the number of points on the hyperelliptic curves 
$y^2=(x^m+a)(x^m+b)(x^m+c)$ and $y^2=x(x^m+a)(x^m+b)(x^m+c)$ over a finite field in terms of the character sums $\varphi_{(l_{1}, l_{2}, l_{3})}$ and $\psi_{(l_{1}, l_{2}, l_{3})}$, 
 and finally obtain expressions in terms of the finite field hypergeometric series.
\end{abstract}
\maketitle
\markboth{P. K. Kewat and R. Kumar}{Generalizations of Jacobsthal Sums}
\section{Introduction And Statement of Results}
Let $p$ be an odd prime, and let $\mathbb{F}_q$ denote the finite field with $q$ elements, where $q=p^e, e\geq 1$. Let $\widehat{\mathbb{F}_q^\times}$ 
denote the group of multiplicative characters $\chi$ on $\mathbb{F}_q^{\times}$. It is well known that $\widehat{\mathbb{F}_q^\times}$ is a cyclic group of order $q-1$. One extends the domain of each character $\chi \in \widehat{\mathbb{F}_q^{\times}}$ 
to all of $\mathbb{F}_q$ by setting $\chi(0):=0$ including the trivial character $\varepsilon$. We denote by $\overline{\chi}$ the character inverse of $\chi$. 
Throughout the paper, the notation $\varphi$ and $\varepsilon$ 
are reserved for quadratic and trivial characters of $\mathbb{F}_q$, respectively. Also, $T$ denotes a fixed generator of $\widehat{\mathbb{F}_q^\times}$. 
For two characters $A$ and $B$ on $\mathbb{F}_q$, 
the binomial coefficient ${A \choose B}$ is defined by
\begin{align}
{A \choose B}:=\frac{B(-1)}{q}J(A,\overline{B})=\frac{B(-1)}{q}\sum_{x \in \mathbb{F}_q}A(x)\overline{B}(1-x),\notag
\end{align}
where $J(A, B)$ denotes the usual Jacobi sum.
Considering the integral representation for the classical hypergeometric series, Greene \cite{greene1987hypergeometric}
defined a finite field analogue of the classical hypergeometric series as follows. Let $A, B, C$ be multiplicative characters on $\mathbb{F}_q$. 
Greene's $_{2}F_1$-finite field
hypergeometric series is defined as
\begin{align}\label{d1.2}
{_{2}}F_1\left(\begin{array}{cc}
                A, & B\\
                 & C
              \end{array}\mid x \right)=\varepsilon(x)\frac{BC(-1)}{q}\sum_{y\in \mathbb{F}_q}B(y)\overline{B}C(1-y)\overline{A}(1-xy).
\end{align}
Greene \cite[Theorem 3.6]{greene1987hypergeometric} expressed the above ${_{2}}F_1$-finite field hypergeometric series in terms of binomial coefficients 
 as given below.
\begin{align}\label{Greene-def-3}
{_{2}}F_1\left(\begin{array}{cc}
                A, & B\\
                 & C
              \end{array}\mid x \right)=\frac{q}{q-1}\sum_{\chi\in \widehat{\mathbb{F}_q^\times}}{A\chi \choose \chi}{B\chi \choose C\chi}\chi(x).
\end{align}
In general, for positive integer $n$, Greene \cite{greene1987hypergeometric} defined the ${_{n+1}}F_n$- finite field hypergeometric series by
\begin{align}\label{greene-def-2}
{_{n+1}}F_n\left(\begin{array}{cccc}
A_0, & A_1, & \ldots, & A_n\\
& B_1, & \ldots, & B_n
\end{array}\mid x \right)
=\frac{q}{q-1}\sum_{\chi\in \widehat{\mathbb{F}_q^\times}}{A_0\chi \choose \chi}{A_1\chi \choose B_1\chi}
              \cdots {A_n\chi \choose B_n\chi}\chi(x),
\end{align}
where $A_0, A_1,\ldots, A_n$ and $B_1, B_2,\ldots, B_n$ are multiplicative characters on $\mathbb{F}_q$.
\par Finite field hypergeometric series were developed mainly to simplify character sum evaluations. In \cite{Ono}, Ono evaluated certain character sums and expressed those sums in terms of finite field hypergeometric series. Williams \cite{williams1979evaluation} evaluated the character sum for certain quadratic polynomials. Recently, Lin and Tu found identities between the twisted Kloosterman sums and the finite field 
hypergeometric series \cite{Lin-Tu}. Recently Many mathematicians evaluated the number of $\mathbb{F}_q$-points of certain algebraic varieties with the 
help of hypergeometric function over finite fields.
(for more detail see \cite{barman2012hypergeometric,barmanhyperelliptic,barman2014hyperelliptic}.
\par Let $n$ be a positive integer and $\varphi$ the quadratic character on $\mathbb{F}_q$. 
If $a\in\mathbb{F}^{\times}_q$, then the Jacobsthal sum $\varphi_{n}(a)$ is defined by
\begin{equation}
 \varphi_{n}(a):=\sum_{x\in\mathbb{F}_q}\varphi(x)\varphi(x^n+a)
\end{equation}
and the modified Jacobsthal sum $\psi_{n}(a)$ is defined by
\begin{equation}
\psi_{n}(a):=\sum_{x\in\mathbb{F}_q}\varphi(x^n+a). 
\end{equation}
In \cite{berndt1979sums}, some of these character sums are evaluated for small values of $n$.  
Let $m, n$ be positive integers. For $a, b\in \mathbb{F}_q^{\times}$,  Sadek \cite{sadek2016jacobs} defined the character sums $\varphi_{m,n}(a,b)$
and $\psi_{m,n}(a,b)$ as generalizations of the Jacobsthal and the modified Jacobsthal sums respectively. He defined these character sums as follows:
\begin{align}\label{e1.6}
 \varphi_{(m, n)}(a, b)&:=\sum_{x\in\mathbb{F}_q}\varphi(x)\varphi(x^m+a)\varphi(x^n+b),\\
 \psi_{(m, n)}(a, b)&:=\sum_{x\in\mathbb{F}_q}\varphi(x^m+a)\varphi(x^n+b).
\end{align}
He studied basic properties of these character sums, and evaluated them when $m$ and $n$ are small powers of $2$. Also, he used these sums to find
the number of $\mathbb{F}_q$- rational points on the hyperelliptic curves $y^2=(x^m+a)(x^m+b)$ and $y^2=x(x^m+a)(x^m+b)$.
\par In this article, we further generalize these character sums. Let $l_1, l_2,\ldots, l_n$ be non-negative integers. For 
$a_1, a_2,\ldots, a_n\in \mathbb{F}_q^{\times}$, we define 
\begin{equation}\label{e1.3}
\varphi_{(l_1,l_2,\ldots,l_n)}(a_1,a_2,\ldots,a_n)=\sum_{x\in\mathbb{F}_q}\varphi(x)\varphi(x^{l_{1}}+a_1)\varphi(x^{l_{2}}+a_2)
\cdots\varphi(x^{l_n}+a_n)
\end{equation}
and
\begin{equation}\label{e1.4}
\psi_{(l_1,l_2,\ldots,l_n)}(a_1,a_2,\ldots,a_n)=\sum_{x\in\mathbb{F}_q}\varphi(x^{l_{1}}+a_1)\varphi(x^{l_{2}}+a_2)
\cdots\varphi(x^{l_n}+a_n).
\end{equation}
We study some basic properties of $\varphi_{(l_1,l_2,\ldots,l_n)}(a_1,a_2,\ldots,a_n)$ and
$\psi_{(l_1,l_2,\ldots,l_n)} (a_1,a_2,\ldots,a_n)$,
and obtain some special values
of these character sums for $n=3$. We express them in terms of Greene's finite field hypergeometric series for certain values of
$l_1,l_2,l_3$. We also express the number of 
$\mathbb{F}_q$-rational points on hyperelliptic curves in terms of these character sums.
More precisely, for $a, b, c\in\mathbb{F}^{\times}_q$ with all distinct, we consider the hyperelliptic curves 
\begin{align}
E_m:& y^2=(x^m+a)(x^m+b)(x^m+c),\\ 
E'_m:& y^2=x(x^m+a)(x^m+b)(x^m+c).
\end{align}
 From \cite{barman2014hyperelliptic}, we recall  that for an hyperelliptic curve $C: y^2=f(x)$ defined over a finite field $\mathbb{F}_q$, the number of points on $C$ including the points at infinity
is given by 
\begin{align}\label{points-EC}
 \#C(\mathbb{F}_q)& = r + \#\{(x, y) \in \mathbb{F}_q^2 : y^2=f(x)\}\notag\\
 &= r + \sum_{x\in \mathbb{F}_q}[1+\varphi(f(x))]\notag\\
 &= r + q + \sum_{x\in \mathbb{F}_q}\varphi(f(x)),
\end{align}
where $r=1$ if $\text{deg}~f(x)$ is odd; and $r=2$ if $\text{deg}~f(x)$ is even. Here, $r$ is the number of $\mathbb{F}_q$-rational points at infinity. 
\par
In the following theorem, we express the number of $\mathbb{F}_q$-rational points on the hyperelliptic curves $E_m$ and $E'_{m}$
in terms of the Greene's finite field hypergeometric series with the help of the character sums $\psi_{(l,m,n)}(a,b,c)$ and $\varphi_{(l,m,n)}(a,b,c)$.
\begin{thm}\label{t1}
Let $\mathbb{F}_q$ be a finite field with $q\equiv 1 \pmod{2m}$. If $\chi_{2m}$ is a multiplicative character of order $2m$ on 
$\mathbb{F}_q$, then we have
\begin{align*}
\#E_m(\mathbb{F}_q)&=r+q+\frac{q^2\varphi(-abc)}{q-1}\sum_{k=0}^{m-1}\sum_{\chi}{\varphi \choose \chi}\overline{\chi}(a)\chi\chi^{k}_m(-b)~
{_2}F_1\left(
\begin{array}{cc}
 \varphi,&\chi\chi^k_m\\
 &\chi\chi^{m+2k}_{2m}
\end{array}|\frac{b}{c}
\right),\\
\#E'_m(\mathbb{F}_q)&=r+q+\frac{q^2\varphi(-abc)}{q-1}\sum_{k=0}^{m-1}\chi^{2k+1}_{2m}(-b)\sum_{\chi}{\varphi \choose \chi}\chi\left(\frac{-b}{a}\right)
{_2}F_1\left(
\begin{array}{cc}
 \varphi,&\chi\chi^{2k+1}_{2m}\\
 &\chi\chi^{m+2k+1}_{2m}
\end{array}| \frac{b}{c}
\right),
\end{align*}
where $r=1$ if $m$ is odd and $r=2$ if $m$ is even.
\end{thm}
Let $\chi_{1},\ldots,\chi_{k}$ be multiplicative characters on $\mathbb{F}_q$. For a given $a\in\mathbb{F}_q$, the character sum $J_a$ is defined by 
\begin{equation}
J_{a}(\chi_{1},\ldots,\chi_{k})=\sum_{c_{1}+\cdots+c_{k}=a}\chi_{1}(c_{1})\ldots\chi_{k}(c_{k}), 
\end{equation}
where the summation is extended over all $k$-tuples $(c_{1},\ldots,c_{k})$ of elements of $\mathbb{F}_q$ with $c_1+\cdots+c_k=a$. If we put 
$a=1$, then $J_a$ is the generalized Jacobi sum (see \cite[Chapter 10]{berndt1998gauss}). Using the generalizations of Jacobsthal sums, we prove the following result.
\begin{thm}\label{t3}
Let $q\equiv 1 \pmod{m}$, where $m=2^n$. For $a_i, b_i\in\mathbb{F}^{\times}_q$, $i=1,\ldots,n$, put $a=a_1+\cdots+a_n$ and $b=b_1+\cdots+b_n$.
Let $\psi,\chi_1,\ldots,\chi_n$ be multiplicative characters on $\mathbb{F}_q$ and $\psi_{m}$ be a character of order $m$ on $\mathbb{F}_q$. Then
\begin{align*}
\sum_{x\in\mathbb{F}_q}\psi(ax^m)\chi_1(b_1-a_{1}x^m)\cdots\chi_n(b_n-a_nx^m)=\sum_{k=0}^{m-1}\psi_2^k(a)J_b(\psi_{m}^{k}\psi,\chi_1,\ldots,
\chi_n). 
\end{align*}
\end{thm}
Let $T$ be a generator of the character group $\widehat{\mathbb{F}^{\times}_q}$. 
In the following results, we express the character sums $\psi_{1,1,1}(a,b,c)$, $\varphi_{(1, 1, 1)}(a, b, c)$ and $\psi_{2,2,2}(a,b,c)$  in terms of Greene's $_2F_1$ -hypergeometric series. 
\begin{thm}\label{t3.4}
Let $a,b,c\in\mathbb{F}^{\times}_q$ with all distinct and $q\equiv 1 \pmod{6}$. If $h\in \mathbb{F}_q^{\times}$ satisfies $3h^2+2(a+b+c)h+ab+bc+ca=0$, then
\begin{align*}
\psi_{(1,1,1)}(a,b,c)=
q\varphi(-3d){_2}F_1
\left(
\begin{array}{cc}
T^{\frac{(q-1)}{6}},&T^{\frac{5(q-1)}{6}}\\
&\varepsilon
\end{array}|-\frac{3^3e}{2^2d^3}
\right),
\end{align*}
where $d=3h+a+b+c$ and $e=h^3+(a+b+c)h^2+(ab+bc+ca)h+abc$.
\end{thm}
\begin{thm}\label{t3.5}
Let $a,b,c\in\mathbb{F}^{\times}_q$ with all distinct and $q\equiv 1 \pmod{6}$. If $h\in \mathbb{F}_q^{\times}$ satisfies $3h^2+2\frac{(ab+bc+ca)}{bc}h+\frac{ab+ca+a^2}{bc}=0$,
then 
 \begin{align*}
\varphi_{(1,1,1)}(a,b,c)=
-1+q\varphi(bc)\varphi(-3f){_2}F_1
\left(
\begin{array}{cc}
T^{\frac{(q-1)}{6}},&T^{\frac{5(q-1)}{6}}\\
&\varepsilon
\end{array}|-\frac{3^3g}{2^2f^3}
\right),
 \end{align*}
where $f=3h+\frac{ab+bc+ca}{bc}$ and $g=h^3+\frac{ab+bc+ca}{bc}h^2+\frac{(ab+ca+a^2)}{bc}h+\frac{a^2}{bc}$
\end{thm}
Finally, we express the character sum $\varphi_{(2, 2, 2)}(a, b, c)$ in terms of sums of Greene's $_2F_1$-hypergeometric series. 
\begin{thm}\label{t3.6}
 Let $a,b,c,d,e,f\in\mathbb{F}^{\times}_q$ be all distinct and $q\equiv 1 \pmod{6}$. If $h\in \mathbb{F}_q^{\times}$ satisfies 
 $3h^2+2\frac{(ab+bc+ca)}{bc}h+\frac{ab+ca+a^2}{bc}=0$ and $3h^2+2(a+b+c)h+ab+bc+ca=0$, then
 \begin{align*}
\psi_{(2,2,2)}(a,b,c)&=-1+
q\varphi(-3d){_2}F_1
\left(
\begin{array}{cc}
T^{\frac{(q-1)}{6}},&T^{\frac{5(q-1)}{6}}\\
&\varepsilon
\end{array}|-\frac{3^3e}{2^2d^3}
\right)\\
&\hspace{1cm} +q\varphi(bc)\varphi(-3f){_2}F_1
\left(
\begin{array}{cc}
T^{\frac{(q-1)}{6}},&T^{\frac{5(q-1)}{6}}\\
&\varepsilon
\end{array}|-\frac{3^3g}{2^2f^3}
\right),
\end{align*}
where $d=3h+a+b+c$, $e=h^3+(a+b+c)h^2+(ab+bc+ca)h+abc$, $f=3h+\frac{ab+bc+ca}{bc}$ and
$g=h^3+\frac{ab+bc+ca}{bc}h^2+\frac{(ab+ca+a^2)}{bc}h+\frac{a^2}{bc}$.
\end{thm}

\section{Preliminaries}
In this section we recall some results and prove some properties of the character sums $\psi_{(l, m, n)}(a, b, c)$ and $\phi_{(l, m, n)}(a, b, c)$ which will be 
used to prove the main results. 
\par 
Let $\delta$ denote the function on multiplicative characters defined by
$$\delta(A)=\left\{
              \begin{array}{ll}
                1, & \hbox{if $A$ is the trivial character;} \\
                0, & \hbox{otherwise.}
              \end{array}
            \right.
$$
We also denote by $\delta$ the function defined on $\mathbb{F}_q$ by 
$$\delta(x)=\left\{
              \begin{array}{ll}
                1, & \hbox{if $x=0$;} \\
                0, & \hbox{if $x\neq 0$.}
              \end{array}
            \right.
$$
\begin{lem}\cite{greene1987hypergeometric} \label{l1.4}
 Let $A$ be a multiplicative character and $a,x\in\mathbb{F}_q$. We have
 \begin{equation}
 A(a+x)=\delta(x)+\frac{qA(a)}{q-1}\sum_{\chi\in\widehat{\mathbb{F}}_q}{{A}\choose{\chi}}\chi\left(\frac{x}{a}\right).
\end{equation}
\end{lem}
We recall the following lemma from \cite{sadek2016jacobs}.
\begin{lem}\cite{sadek2016jacobs}\label{l1.1}
Let $f$ be a function from $\mathbb{F}_q$ to $\mathbb{C}$. We have
\begin{equation*}
\sum_{x\in\mathbb{F}_q}\varphi(x)f(x)=\sum_{x\in\mathbb{F}_q}f(x^2)-\sum_{x\in\mathbb{F}_q}f(x). 
\end{equation*}
In particular, we have $\psi_{(2m,2n)}(a,b)=\psi_{(m,n)}(a,b)+\varphi_{(m,n)}(a,b)$, and $\psi_{(2l,2m,2n)}(a,b,c)=\psi_{(l,m,n)}(a,b,c)
+\varphi_{(l,m,n)}(a,b,c)$.
\end{lem}
In the following lemma, we link the character sum to Jacobi sum. 
\begin{lem}\label{c1.4}
Let $a\in\mathbb{F}^{\times}_q$, and $\psi,\chi$ and $\rho$ be characters on $\mathbb{F}_q$. Then
\begin{align*}
\sum_{x\in\mathbb{F}_q}\psi(2x^2)\chi(x^2)\rho(1+ax^2)=\psi(-1)\psi\chi(a^{-1})J(\psi,\chi,\rho)+\psi(-1)\psi\chi\varphi(a^{-1})
J(\psi,\chi\varphi,\rho),
\end{align*}
where $J(\psi,\chi,\rho)$ is the Jacobi sum.
\end{lem}
\begin{proof}
By setting $f(x)=\psi(2x)\chi(x)\rho(1+ax)$ in Lemma \ref{l1.1}, we have
\begin{equation*}
\sum_{x\in\mathbb{F}_q}\psi(2x^2)\chi(x^2)\rho(1+ax^2)=\sum_{x\in\mathbb{F}_q}\psi(2x)\chi(x)\rho(1+ax)
+\sum_{x\in\mathbb{F}_q}\varphi(x)\psi(2x)\chi(x)\rho(1+ax).
\end{equation*}
Now using the bijective transformation $x\mapsto-a^{-1}x$, we have
\begin{align*}
\sum_{x\in\mathbb{F}_q}\psi(2x^2)\chi(x^2)\rho(1+ax^2)=&\psi(-1)\psi\chi(a^{-1})\sum_{x\in\mathbb{F}_q}\psi(2x)\chi(-x)\rho(1-x)\\
 +&\psi(-1)\psi\chi\varphi(a^{-1})\sum_{x\in\mathbb{F}_q}\psi(2x)\chi\varphi(-x)\rho(1-x).
\end{align*}
From the definition of the Jacobi sum, we have the desired result.
\end{proof} 
\begin{lem}\label{c15}
 Let $\psi,\chi,\rho$ be characters on $\mathbb{F}_q$ with $q\equiv 1 \pmod{4}$ and 
$\chi_{4}$ be a character of order $4$. Then we have
\begin{equation*}
 \sum_{x\in\mathbb{F}_q}\psi(2x^4)\chi(x^4)\rho(1+ax^4)=\psi(-1)\sum_{k=0}^{3}\psi\chi\chi^{k}_{4}(a^{-1})J(\psi,\chi\chi^{k}_{4},\rho).
\end{equation*}
\end{lem}
\begin{proof}
Lemma \ref{l1.1} yields
\begin{align*}
 \sum_{x\in\mathbb{F}_q}\psi(2x^4)\chi(x^4)\rho(1+ax^4)&=\sum_{x\in\mathbb{F}_q}\psi(2x^2)\chi(x^2)\rho(1+ax^2)
 +\sum_{x\in\mathbb{F}_q}\varphi(x)\psi(2x^2)\chi(x^2)\rho(1+ax^2)\\
 &=\sum_{x\in\mathbb{F}_q}\psi(2x^2)\chi(x^2)\rho(1+ax^2)+\sum_{x\in\mathbb{F}_q}\psi(2x^2)\chi_4\chi(x^2)\rho(1+ax^2).
\end{align*}
Now using Lemma \ref{c1.4}, we have
\begin{align*}
\sum_{x\in\mathbb{F}_q}\psi(2x^4)\chi(x^4)\rho(1+ax^4)=\psi(-1)\sum_{k=0}^{3}\psi\chi\chi^{k}_{4}(a^{-1})J(\psi,\chi\chi^{k}_{4},\rho).
\end{align*}
\end{proof}
Now, we prove the next result by using induction.
\begin{lem}
Let $a\in\mathbb{F}_q^{\times}$ and $m=2^n$, where $n\in\mathbb{N}$ such that $q\equiv 1 \pmod{m}$. Let $\psi,\chi,\rho$ be characters on $\mathbb{F}_q$, and 
$\chi_{m}$ be a character of order $m$ on $\mathbb{F}_q$. Then we have
\begin{align*}
\sum_{x\in\mathbb{F}_q}\psi(2x^m)\chi(x^m)\rho(1+ax^m)=\psi(-1)\sum_{k=0}^{m-1}\psi\chi\chi^{k}_{m}(a^{-1})J(\psi,\chi\chi^{k}_{m},\rho). 
\end{align*}
\end{lem}
\begin{proof}
In view of Lemma \ref{c1.4} and Lemma \ref{c15}, we will prove this lemma by induction method. 
\end{proof}
In the following proposition we prove some basic properties of the character sums\\
$\psi_{(l_1,\,l_2,\,\ldots,\,l_n)}(a_1,\,a_2,\,\ldots,a_n)$ and $\varphi_{(l_1,\,l_2,\,\ldots,\,l_n)}(a_1,\,a_2,\,\ldots,\,a_n)$.  
\begin{prop} For $a_1,a_2,\ldots,a_n\in\mathbb{F}^{\times}_q$, we have 
\begin{enumerate}
\item Let $l_1,l_2,\ldots,l_n$ be non negative integer, then
$
\psi_{(l_{1},\,l_{2},\,\ldots,\,l_{n})}(a_{1},\,a_{2},\ldots,\,a_{n})=\\
\psi_{(l_{2},\,l_{3},\,\ldots,\,l_{1})}(a_{2},\,a_{3},\,\ldots,\,a_{1})=\ldots=\psi_{(l_{n},\,l_{1},\,\ldots,\,l_{(n-1)})}(a_{n},\,a_{1},\,\ldots,\,a_{n-1});\\
\varphi_{(l_{1},\,l_{2},\,\ldots,\,l_{n})}(a_{1},\,a_{2},\,\ldots,\,a_{n})=
\psi_{(l_{2},\,l_{3},\,\ldots,\,l_{1})}(a_{2},\,a_{3},\,\ldots,\,a_{1})=\\
\cdots=\psi_{(l_{n},\,l_{1},\,\ldots,\,l_{(n-1)})}(a_{n},\,a_{1},\,\ldots,\,a_{n-1}).\\
$
\item Let $I=\{i_1,i_2,\ldots,i_k\}\subset\{1,2,\ldots,n\}$. If $l_r=0$ for all $r\in I$ then
$
\psi_{(l_{1},\ldots,l_{n})}(a_{1},\ldots,a_{n})=\prod_{r\in I}\varphi(1+a_r)\psi_{(l_{1},
\ldots,\hat{l}_{i_1},\ldots,\hat{l}_{i_{k}},\ldots l_{n})}(a_{1},\ldots,\hat{a}_{i_1},\ldots,\hat{a}_{i_{k}},\ldots,a_{n});\\
\varphi_{(l_{1},\ldots,l_{n})}(a_{1},\ldots,a_{n})=\prod_{r\in I}\varphi(1+a_r)\varphi_{(l_{1},
\ldots,\hat{l}_{i_1},\ldots,\hat{l}_{i_{k}},\ldots l_{n})}(a_{1},\ldots,\hat{a}_{i_1},\ldots,\hat{a}_{i_{k}},\ldots,a_{n}),
$
where $\hat{l}_{i_{k}}$ represents the absentia of $l_{i_{k}}$ in $\{l_{1},\ldots,l_{i_1},\ldots,l_{i_{k}},
\ldots l_{n}\}$.\\
\item
$
 \psi_{(l_{1},\,l_{2},\,\ldots,l_{n})}(a^{l_1}_{1},\,a_{2},\,\ldots,a_{n})=\varphi(a^{l_1+l_2+\cdots+l_n})
 \psi_{(l_{1},\,l_{2},\,\ldots,l_{n})}(1,\frac{a_2}{a_{1}^{l_2}},\ldots\frac{a_n}{a_{1}^{l_n}});\\
 \varphi_{(l_{1},\,l_{2},\,\ldots,l_{n})}(a^{l_1}_{1},\,a_{2},\,\ldots,a_{n})=\varphi(a^{l_1+l_2+\cdots+l_n+1})
 \varphi_{(l_{1},\,l_{2},\,\ldots,l_{n})}(1,\frac{a_2}{a_{1}^{l_2}},\ldots\frac{a_n}{a_{1}^{l_n}}).\\
 $
 \item If $l_1=l_2=\cdots=l_n=l$ and $a_1=a_2=\cdots=a_n=a$, then
\[
 \psi_{(l,\,l,\,\ldots,\,l)}(a,\,a,\,\ldots,\,a)=\varphi_{(l,\,l,\,\ldots,\,l)}(a,\,a,\,\ldots,\,a)=
 \begin{cases}
\psi_{l}(a), &\mbox{if n is odd};\\
q-1, &\mbox{if n is even}.
\end{cases}
\]
\end{enumerate}
\end{prop}
\begin{proof}
We readily obtain $(1)$ and $(2)$ from the definition. For $(3)$, we have the following identity
\begin{align*}
 \psi_{(l_1,l_2,\ldots,l_n)}(a_1^{l_1},a_{2},\ldots,a_n)&=\sum_{x\in\mathbb{F}_q}\varphi(x^{l_1}+a_1^{l_1})
 \varphi(x^{l_2}+a_2)\cdots\varphi(x^{l_n}+a_n)\\
&=\varphi(a_1^{l})\sum_{x\in\mathbb{F}_q}\varphi\left(\frac{x^{l_1}}{a_1^{l_1}}+1\right)\varphi(x^{l_2}+a_2)\cdots\varphi(x^{l_n}+a_n).
 \end{align*}
If we replace $x$ by $a_1x$ in the second equality, we have the desired result. The proof of the second part is similar to the proof of 
first part.
Again, we readily obtain $(4)$ from the definition of $\psi_{(l_1,\,l_2,\,\ldots,\,l_n)}(a_1,\,a_{2},\,\ldots,\,a_n)$
and $\varphi_{(l_1,\,l_2,\,\ldots,\,l_n)}(a_1,\,a_{2},\,\ldots,\,a_n)$.
\end{proof}
We finally recall two theorems from \cite{barmanhyperelliptic}.
Let $a,b\in\mathbb{F}^{\times}_q$. In \cite{barmanhyperelliptic}, Barman and Kalita expressed
the number of $\mathbb{F}_q$-points on the hyperelliptic curve $E_d:y^2=x^d+ax^{d-1}+b$ in terms of Greene's $_{d-1}F_{d-2}$-hypergeometric series as stated below.
\begin{thm}\cite{barmanhyperelliptic}\label{bt1}
 Let $p$ be a prime, and $q$ be a power of $p$. If $d\geq2$ is an even 
 integer and $q\equiv 1 \pmod{2d(d-1)}$, then
\begin{align*}
 &\#\{(x, y)\in \mathbb{F}_q: y^2=x^d+ax^{d-1}+b\}\\
 &=q+\varphi(b)+q^{\frac{d}{2}}\varphi(d-1)
 _{d}F_{d-1}
\left(
\begin{array}{ccccc}
\varphi,&\varepsilon,&\chi,&\chi^2,\ldots,\chi^{\frac{d-2}{2}},&\chi^{\frac{d+2}{2}},\ldots,\chi^{d-1}\\
&\varphi,&\psi,&\psi^3,\ldots,\psi^{d-3},&\psi^{d+1},\ldots,\psi^{2d-3}
\end{array}|\alpha\right),
\end{align*}
where $\chi$ and $\psi$ are characters of order $d$ and $2(d-1)$, respectively; and $\alpha=\frac{bd^d}{a^d(d-1)^{d-1}}$.
\end{thm}
\begin{thm}\cite{barmanhyperelliptic}\label{bt2}
Let $p$ be a prime, and $q$ be a power of $p$. If $d\geq 3$ is an odd 
 integer and $q\equiv 1 \pmod{2d(d-1)}$, then
\begin{align*}
 &\#\{(x, y)\in \mathbb{F}_q: y^2=x^d+ax^{d-1}+b\}\\ 
 &=q+q^{\frac{d-1}{2}}\varphi(-ad)
 _{d-1}F_{d-2}
\left(
\begin{array}{ccccc}
\eta,&\eta^3,&\eta^5,\ldots,\eta^{d-2},&\eta^{d+2},\ldots,\eta^{2d-3},&\eta^{2d-1}\\
&\rho,&\rho^2,\ldots,\rho^{\frac{d-3}{2}},&\rho^{\frac{d+1}{2}},\ldots,\rho^{d-2},&\varepsilon
\end{array}|-\alpha\right),
\end{align*}
where $\eta$ and $\rho$ are characters of order $2d$ and $(d-1)$, respectively; and $\alpha=\frac{bd^d}{a^d(d-1)^{d-1}}$.
\end{thm}

\section{Proof of the Results}
In this section, we will prove Theorem \ref{t1} and Theorem \ref{t3}. To prove Theorem \ref{t1},
 first we prove the following lemma.
 \begin{lem}\label{l3.5}
 Let $m\in\mathbb{N}$ and $q\equiv 1 \pmod{2m}$. Let $\chi_{2m}$ be a character of order $2m$ over $\mathbb{F}_q$, then we have
 \begin{equation*}
  \psi_{(m,\ldots,m)}(a_1,\ldots,a_r)=\sum_{x\in\mathbb{F}_q}\varphi(x+a_{1})\cdots\varphi(x+a_r)\sum_{k=0}^{m-1}\chi_{2m}^{2k}(x)
 \end{equation*}
and
\begin{equation*}
 \varphi_{(m,\ldots,m)}(a_1,\ldots,a_r)=\sum_{x\in\mathbb{F}_q}\varphi(x+a_{1})\cdots\varphi(x+a_r)\sum_{k=0}^{m-1}\chi_{2m}^{2k+1}(x).
\end{equation*}
\end{lem}
\begin{proof}
Note that $\chi^{2}_{2m}$ is a character of order $m$. Therefore,
\begin{eqnarray}\label{e3.5}
\sum_{k=0}^{m-1}(\chi^{2}_{2m})^{k}(x)=\begin{cases}
                               1~\mbox{if}~ x=0\\
                               m~\mbox{if $x$ is an $m$-th power in $\mathbb{F}^{\times}_q$}\\
                               0~\mbox{otherwise}.
                               \end{cases}                               
\end{eqnarray}
Also $\chi_{2m}$ is a character of order $2m$, thus
\begin{eqnarray}\label{e3.6}
 \sum_{k=0}^{2m-1}(\chi_{2m})^{k}(x)=\begin{cases}
                               1~\mbox{if}~ x=0\\
                               2m~\mbox{if $x$ is an $2m$-th power in $\mathbb{F}^{\times}_q$}\\
                               0~\mbox{otherwise}.
                               \end{cases}                               
\end{eqnarray}
Since
\begin{equation*}
 \psi_{(m,\ldots,m)}(a_1,\ldots,a_r)=\sum_{x\in\mathbb{F}_q}\varphi(x^m+a_1)\cdots\varphi(x^m+a_r).
\end{equation*}
We use Equation \ref{e3.5} to write
\begin{equation}\label{e3.7}
 \psi_{(m,\ldots,m)}(a_1,\ldots,a_r)=\sum_{x\in\mathbb{F}_q}\varphi(x+a_{1})\cdots\varphi(x+a_r)\sum_{k=0}^{m-1}\chi_{2m}^{2k}(x).
\end{equation}
From Equation \ref{e3.6}, we obtain
\begin{equation}\label{e3.8}
 \psi_{(2m,\ldots,2m)}(a_1,\ldots,a_r)=\sum_{x\in\mathbb{F}_q}\varphi(x+a_{1})\cdots\varphi(x+a_r)\sum_{k=0}^{2m-1}\chi_{2m}^{k}(x).
\end{equation}
Now we recall that
\begin{equation}\label{e3.9}
 \varphi_{(m,\ldots,m)}(a_1,\ldots,a_r)=\psi_{(2m,\ldots,2m)}(a_1,\ldots,a_r)-\psi_{(m,\ldots,m)}(a_1,\ldots,a_r).
\end{equation}
Finally, in view of Equations \ref{e3.7}, \ref{e3.8} and \ref{e3.9}, we conclude that
\begin{align*}
 \varphi_{(m,\ldots,m)}(a_1,\ldots,a_r)=&\sum_{x\in\mathbb{F}_q}\varphi(x+a_{1})\cdots\varphi(x+a_r)\sum_{k=0}^{2m-1}\chi_{2m}^{k}(x)
 \\-&\sum_{x\in\mathbb{F}_q}\varphi(x+a_{1})\cdots\varphi(x+a_r)\sum_{k=0}^{m-1}\chi_{2m}^{2k}(x)\\
 \varphi_{(m,\ldots,m)}(a_1,\ldots,a_r)=&\sum_{x\in\mathbb{F}_q}\varphi(x+a_{1})\cdots\varphi(x+a_r)
\sum_{k=0}^{m-1}\chi_{2m}^{2k+1}(x).
\end{align*}
\end{proof}
\begin{proof}[{\bf Proof of Theorem} \ref{t1}]
We use Lemma \ref{l3.5} to write
\begin{equation*}
 \psi_{(m,m,m)}(a,b,c)=\sum_{x\in\mathbb{F}_q}\varphi(x+a)\varphi(x+b)\varphi(x+c)\sum_{k=0}^{m-1}\chi_{m}^{k}(x).
\end{equation*}
Using Lemma \ref{l1.4}, and then from definition of $\delta$, we deduce that
\begin{align}\label{mpe1}
\psi_{(m,m,m)}(a,b,c)=&\frac{q\varphi(a)}{q-1}\sum_{k=0}^{m-1}\sum_\chi{\varphi \choose\chi}\sum_{x}\varphi(x+b)\varphi(x+c)
\chi\left(\frac{x}{a}\right)\chi_m^k(x).
\end{align}
If we replace $x$ by $-bx$ in Equation \ref{mpe1}, we obtain
\begin{align*}
 \psi_{(m,m,m)}(a,b,c)=&\frac{q\varphi(a)}{q-1}\sum_{k=0}^{m-1}\sum_\chi{\varphi \choose\chi}\sum_{x}\varphi(-bx+b)\varphi(-bx+c)
\chi\left(\frac{-bx}{a}\right)\chi_m^k(-bx)
\end{align*}
\begin{align*}
=&\frac{q\varphi(abc)}{q-1}\sum_{k=0}^{m-1}\chi_m^{k}(-b)\sum_\chi{\varphi \choose\chi}\chi\left(\frac{-b}{a}\right)\sum_{x}\varphi(1-x)
\varphi(1-\frac{b}{c}x)\chi\chi_m^k(x)\\
=&\frac{q\varphi(abc)}{q-1}\sum_{k=0}^{m-1}\chi_m^{k}(-b)\sum_\chi{\varphi \choose\chi}\chi\left(\frac{-b}{a}\right)\sum_{x}
\chi\chi_m^k(x)
\overline{\chi\chi_m^k}\chi\chi_{2m}^{m+2k}(1-x)\varphi(1-\frac{b}{c}x).
\end{align*}
From Equation \ref{d1.2}, the inner sum can be written in terms of Greene's finite field hypergeometric series as given below
\begin{align}\label{e3.12}
 \psi_{(m,m,m)}(a,b,c)=&\frac{q^2\varphi(-abc)}{q-1}\sum_{k=0}^{m-1}\chi_m^{k}(-b)\sum_\chi{\varphi \choose\chi}\chi\left(\frac{-b}{a}\right)
 {_2}F_1\left(
\begin{array}{cc}
 \varphi,&\chi\chi^k_m\\
 &\chi\chi^{m+2k}_{2m}
\end{array}|\frac{b}{c}
\right).
\end{align}
We use Equations \ref{points-EC} and \ref{e1.4} to write number of $\mathbb{F}_q$-points on $E_m$ as follows
\begin{equation}\label{e3.13}
 \#E_m(\mathbb{F}_q)=r+q+\psi_{m,m,m}(a,b,c).
\end{equation}
Using Equations \ref{e3.12} and \ref{e3.13}, the number of $\mathbb{F}_q$-rational points on algebraic curves $E_m$ is express as
\begin{align*}
\#E_m(\mathbb{F}_q)&=r+q+\frac{q^2\varphi(-abc)}{q-1}\sum_{k=0}^{m-1}\sum_{\chi}{\varphi \choose\chi}\overline{\chi}(a)\chi\chi^{k}_m(-b)~
{_2}F_1\left(
\begin{array}{cc}
 \varphi,&\chi\chi^k_m\\
 &\chi\chi^{m+2k}_{2m}
\end{array}|\frac{b}{c}
\right).
\end{align*}
Next, we prove second equality. The proof is similar to first equality and so we omit some steps.
First we use Lemma \ref{l3.5} and then using Lemma \ref{l1.4} and definition of $\delta$, we obtain
\begin{align}\label{e3.90}
\psi_{(m,m,m)}(a,b,c)=&\frac{q\varphi(a)}{q-1}\sum_{k=0}^{m-1}\sum_\chi{\varphi \choose\chi}\sum_{x}\varphi(x+b)\varphi(x+c)
\chi\left(\frac{x}{a}\right)\chi_{2m}^{2k+1}(x).
\end{align}
Replace $x$ by $-bx$ in Equation \ref{e3.90}, we obtain
\begin{align*}
 \psi_{(m,m,m)}(a,b,c)=&\frac{q\varphi(a)}{q-1}\sum_{k=0}^{m-1}\sum_\chi{\varphi \choose\chi}\sum_{x}\varphi(-bx+b)\varphi(-bx+c)
\chi\left(\frac{-bx}{a}\right)\chi_{2m}^{2k+1}(-bx)
\end{align*}
\begin{align*}
=\frac{q\varphi(abc)}{q-1}\sum_{k=0}^{m-1}\chi_{2m}^{2k+1}(-b)\sum_\chi{\varphi \choose\chi}\chi\left(\frac{-b}{a}\right)\sum_{x}
\chi\chi_{2m}^{2k+1}(x)
\overline{\chi\chi_{2m}^{2k+1}}\chi\chi_{2m}^{m+2k+1}(1-x)\varphi(1-\frac{b}{c}x).
\end{align*}
From Equation \ref{d1.2}, the inner sum can be written as
\begin{align}\label{e3.14}
 \psi_{(m,m,m)}(a,b,c)=&\frac{q^2\varphi(-abc)}{q-1}\sum_{k=0}^{m-1}\chi_{2m}^{2k+1}(-b)\sum_\chi{\varphi \choose\chi}\chi\left(\frac{-b}{a}\right)
 {_2}F_1\left(
\begin{array}{cc}
 \varphi,&\chi\chi^{2k+1}_{2m}\\
 &\chi\chi^{m+2k+1}_{2m}
\end{array}|\frac{b}{c}
\right).
\end{align}
Now, we use Equations \ref{points-EC} and \ref{e1.3} to write number of $\mathbb{F}_q$- points on $E^{'}_m$ as follows
\begin{equation}\label{e3.15}
 \#E^{'}_m(\mathbb{F}_q)=r+q+\varphi_{m,m,m}(a,b,c).
\end{equation}
Using Equations \ref{e3.14}, \ref{e3.15} and simplifying, we obtain
\begin{align*}
 \#E'_m(\mathbb{F}_q)&=r+q+\frac{q^2\varphi(-abc)}{q-1}\sum_{k=0}^{m-1}\chi^{2k+1}_{2m}(-b)\sum_{\chi}{\varphi \choose\chi}\chi\left(\frac{-b}{a}\right)
{_2}F_1\left(
\begin{array}{cc}
 \varphi,&\chi\chi^{2k+1}_{2m}\\
 &\chi\chi^{m+2k+1}_{2m}
\end{array}| \frac{b}{c}
\right).
\end{align*}
\end{proof}
For proving Theorem \ref{t3}, first we prove following Lemmas. Then using simple induction method we will prove Theorem \ref{t3}.
\begin{lem}\label{c1.7}
For $\psi,\chi_1,\ldots,\chi_n$ are characters of $\mathbb{F}_q$ and $a_i,b_i\in\mathbb{F}^{\times}_q$ $\forall i=1,\ldots,n$
with $a=a_1+\cdots+a_n$ and $b=b_1+\cdots+b_n$. Then
\begin{align*}
 \sum_{x\in\mathbb{F}_q}\psi(ax^2)\chi_1(b_1-a_{1}x^2)\cdots\chi_n(b_n-a_nx^2)=J_b(\psi,\chi_1,\ldots,\chi_n)+\varphi(a)J_b(\varphi\psi,\chi_1,
 \ldots,\chi_n).
\end{align*}
\end{lem}
\begin{proof}
From Lemma \ref{l1.1}, we obtain
\begin{align*}
\sum_{x\in\mathbb{F}_q}\psi(ax^2)\chi_1(b_1-a_{1}x^2)\cdots\chi_n(b_n-a_nx^2)=&\sum_{x\in\mathbb{F}_q}\psi(ax)\chi_1(b_1-a_{1}x)
\cdots\chi_n(b_n-a_nx)\\
+&\sum_{x\in\mathbb{F}_q}\varphi(x)\psi(ax)\chi_1(b_1-a_{1}x)\cdots\chi_n(b_n-a_nx).
\end{align*}
From the definition of Jacobi sum, we have
\begin{align*}
\sum_{x\in\mathbb{F}_q}\psi(ax^2)\chi_1(b_1-a_{1}x^2)\cdots\chi_n(b_n-a_nx^2)=&J_b(\psi,\chi_1,\cdots,\chi_n)+\varphi(a)J_{b}(\varphi\psi,\chi_1,
\cdots,\chi_n). 
\end{align*}
\end{proof}
\begin{lem}\label{c1.8}
Let $\psi,\chi_1,\ldots,\chi_n$ are characters of $\mathbb{F}_q$ such that $q\equiv 1 \pmod{4}$ and $a_i,b_i\in\mathbb{F}^{\times}_q$ $\forall i=1,\ldots,n$
with $a=a_1+\cdots+a_n$ and $b=b_1+\cdots+b_n$. Let $\psi_2$ and $\psi_{4}$ are characters of order $2$ and $4$ respectively, then
\begin{align*}
\sum_{x\in\mathbb{F}_q}\psi(ax^4)\chi_1(b_1-a_{1}x^4)\cdots\chi_n(b_n-a_nx^4)=
\sum_{k=0}^{3}\psi_{2}^{k}(a)J_b(\psi_4^k\psi,\chi_1,\cdots,\chi_n).
\end{align*}
\end{lem}
\begin{proof}
 Proof is similar to Lemma \ref{c1.7}.
\end{proof}
\begin{proof}[{\bf Proof of Theorem} \ref{t3}]
In view of Lemmas \ref{c1.7} and \ref{c1.8}, we prove Theorem \ref{t3} by simple induction method.
\end{proof}
\section{Evaluation of character sums $\psi_{(1,1,1)}(a,b,c)$, $\varphi_{(1,1,1)}(a,b,c)$ and $\psi_{(2,2,2)}(a,b,c)$}
In this section, we will prove Theorems
\ref{t3.4}, \ref{t3.5} and \ref{t3.6}. These theorems give the special values of 
character sums $\psi_{(l,m,n)}(a,b,c)$ and $\varphi_{(l,m,n)}(a,b,c)$ in terms of hypergeometric function.
\begin{proof}[{\bf Proof of Theorem} \ref{t3.4}]
 A change of variable $(x,y)\mapsto(x+h,y)$ takes the elliptic curve $y^2=x^3+(a+b+c)x^2+(ab+bc+ca)x+abc$ to $y^2=x^3+dx^2+e$, where 
$d=3h+a+b+c$ and $e=h^3+(a+b+c)h^2+(ab+bc+ca)h+abc$. Clearly, we have
\begin{align*}
 &\#\{(x,y)\in\mathbb{F}^{2}_q:y^2=x^3+dx^2+e\}\\
 &=\#\{(x,y)\in\mathbb{F}^{2}_q:y^2=x^3+(a+b+c)x^2+(ab+bc+ca)x+abc\}\\
 &=\sum_{x\in\mathbb{F}_q}\sum_{\chi^2=\varepsilon}\chi(x^3+(a+b+c)x^2+(ab+bc+ca)x+abc)\\
&=q+\sum_{x\in\mathbb{F}_q}\varphi(x^3+(a+b+c)x^2+(ab+bc+ca)x+abc).
\end{align*}
In view of Equation \ref{e1.4}, we have
\begin{align*}
\#\{(x,y)\in\mathbb{F}^{2}_q:y^2=x^3+dx^2+e\}=q+\psi_{(1,1,1)}(a,b,c).
\end{align*}
From Theorem \ref{bt2}, we have a desired result.
\end{proof}
\begin{proof}[{\bf Proof of Theorem} \ref{t3.5}]
From Equation \ref{e1.3}, we have
\begin{align*}
\varphi_{(1,1,1)}(a,b,c)=&\sum_{x\in\mathbb{F}_q}\varphi(x)\varphi(x+a)\varphi(x+b)\varphi(x+c)\\
=&\sum_{x\in\mathbb{F}^{\times}_q}\varphi(1+\frac{a}{x})\varphi(1+\frac{b}{x})\varphi(1+\frac{c}{x}).
\end{align*}
 If we replace $x$ by $\frac{a}{x}$, we deduce that
\begin{align*}
 \varphi_{(1,1,1)}(a,b,c)=&\sum_{x\in\mathbb{F}^{\times}_q}\varphi(1+x)\varphi(1+\frac{b}{b}x)\varphi(1+\frac{c}{a}x)
 \end{align*}
 \begin{align*}
 =&\varphi(bc)\sum_{x\in\mathbb{F}^{\times}_q}\varphi\left(x^3+\left(\frac{ab+bc+ca}{bc}\right)x^2+\left(\frac{ab+ac+a^2}{bc}\right)x+
 \frac{a^2}{bc}\right)\\
 =&\varphi(bc)\sum_{x\in\mathbb{F}_q}\varphi\left(x^3+\left(\frac{ab+bc+ca}{bc}\right)x^2+\left(\frac{ab+ac+a^2}{bc}\right)x+
 \frac{a^2}{bc}\right)-1.
\end{align*}
We now simplify this expression using the same techniques as in the proof of Theorem \ref{t3.4}
\begin{align*}
\#\{(x,y)\in\mathbb{F}^{2}_q:y^2=x^3+fx^2+g\}=&q+\frac{\varphi_{(1,1,1)}(a,b,c)+1}{\varphi(bc)}\\
\varphi_{(1,1,1)}(a,b,c)=&-1+q\varphi(bc)\varphi(-3f){_2}F_1
\left(
\begin{array}{cc}
T^{\frac{(q-1)}{6}},&T^{\frac{5(q-1)}{6}}\\
&\varepsilon
\end{array}|-\frac{3^3g}{2^2f^3}
\right),
\end{align*}
where $f=3h+\frac{ab+bc+ca}{bc}$ and $g=h^3+\frac{ab+bc+ca}{bc}h^2+\frac{(ab+ca+a^2)}{bc}h+\frac{a^2}{bc}$.
\end{proof}
\begin{proof}[{\bf Proof of Theorem} \ref{t3.6}] In view of Lemma \ref{l1.1} and using Theorems \ref{t3.4} and \ref{t3.5}, we obtain desired result.
\end{proof}
\bibliographystyle{plain} 
\bibliography{ref}
\end{document}